\documentclass{article}
\usepackage[margin=1in]{geometry}

\usepackage{amssymb}
\usepackage{url}

\usepackage{amsmath}
\usepackage{amsthm}
\usepackage{hyperref}
\newtheorem{theo}{Theorem}[section]
\newtheorem{defn}[theo]{Definition}
\newtheorem{rem}[theo]{Remark}
\newtheorem{exmp}[theo]{Example}
\newtheorem{prop}[theo]{Proposition}

\newtheorem{cor}[theo]{Corollary}

\newcommand{\R}{\mathbb{R}}
\newcommand{\Rmax}{\R_{\max}}
\newcommand{\Rmin}{\R_{\min}}
\newcommand{\bomega}{\bigoplus_{\omega}}
\newcommand{\idxset}{\mathcal{I}}
\newcommand{\solset}{\mathcal{S}}

\newcommand{\tred}[1]{\textcolor{red}{#1}}
\newcommand{\ceils}[1]{\lceil#1\rceil}

\makeatletter
\title{Solving Linear Equations Over Maxmin-$\omega$ Systems}

\author{Muhammad Syifa'ul Mufid\footnote{Department of Mathematics, Institut Teknologi Sepuluh Nopember, Jl Arif Rahman Hakim, Surabaya,
            60111, Indonesia. Email: syifaul.mufid@matematika.its.ac.id}\and 
            Ebrahim Patel\footnote{The London Interdisciplinary School, Whitechapel Road, London, E1 1EW,  UK. Email: ebrahimp@yahoo.com}\and
            Serge\u{\i} Sergeev\footnote{University of Birmingham, School of Mathematics, Edgbaston B15 2TT, UK. Email: s.sergeev@bham.ac.uk}}

\date{}

\begin{document}
\maketitle
\makeatother

\begin{abstract}
Maxmin-$\omega$ dynamical systems were previously introduced as an ``all-in-one package'' that can yield a solely min-plus, a solely max-plus, or a max-min-plus dynamical system by varying a parameter $\omega\in(0,1]$. With such systems in mind, it is natural to introduce and consider maxmin-$\omega$ linear systems of equations of the type $A\otimes_{\omega} x=b$. However, to our knowledge, such maxmin-$\omega$ linear systems have not been studied before and in this paper we present an approach to solve them. We show that the problem can be simplified by performing normalization and then generating a ``canonical'' matrix which we call the principal order matrix. Instead of directly trying to find the solutions, we search the possible solution indices which can be identified using the principal order matrix and the parameter $\omega$. The fully active solutions are then immediately obtained from these solution indices. With the fully active solutions at hand, we then present the method to find other solutions by applying a relaxation, i.e., increasing or decreasing some components of fully active solutions. This approach can be seen as a generalization of an approach that could be applied to solve max-plus or min-plus linear systems.
Our results also shed more light on an unusual feature of maxmin-$\omega$ linear systems, which, unlike in the usual linear algebra, can have a finite number of solutions in the case where their solution is non-unique.

\end{abstract}

\section{Introduction}
Discrete event systems (DES) are a class of dynamic systems in which the change of state is governed by events \cite{baccelli92,cassandras2008}. Typical DES applications include manufacturing systems \cite{imaev2008}, railway networks \cite{Heid2006}, queues \cite{Zimmerman2007}, and urban traffic systems \cite{dotoli2006}.  Since events in a DES occur in some order, there is often a need for synchronization, which could mean, for example, that a new event (or process) starts as soon as all preceding processes have finished.  Mathematically, this type of synchronization corresponds to the operation ``maximization" on the earliest starting time of an event and, since the
finishing time of an event equals the starting time plus
the duration, the operation ``addition" corresponds to the duration of events. This leads to a description that is linear in
the so-called max-plus algebra (today also called tropical linear algebra).



A more specific motivation for solving max-plus problems comes from operational applications such as the following job-scheduling task \cite{cuninghamegreen,Butkovic2010}. Suppose that products $P_1,\ldots,P_m$ are prepared using $n$ machines (processors), every machine contributing to the completion of each product by producing a component. It is assumed that each machine can work for all products simultaneously and that all these actions on a machine start as soon as the machine starts to work. Let $A(i,j)$ be the duration of the work of the $j$th machine needed to complete the component for $P_i$ $(i = 1,...,m;j = 1,...,n)$. If this interaction is not required for some $i$ and $j$ then $A(i,j)$ is set to $-\infty$. The matrix $A = (A(i,j))$ is called the production matrix. Let us denote by $x_j$ the starting time of the $j$-th machine ($j = 1,\ldots,n$). Then all components for $P_i$ ($i = 1,...,m$) will be ready at time
\[\max\{x_1 +A(i,1),\ldots,x_n +A(i,n)\}.\]
Hence if $b_1,\ldots, b_m$ are given completion times then the starting times have to satisfy the following system of equations.
$$\max(x_1 +A(i,1),\ldots,x_n +A(i,n))=b_i \text{ for all } i=1,...,m.$$
Using max-plus algebra this system can be written in a compact form as a system of linear equations:
\begin{equation}\label{equ:maxpluslinear}
    A\otimes x=b
\end{equation}
It is easy to decide if a system of the form \eqref{equ:maxpluslinear} is solvable \cite{Butkovic2010} and, in the case of solvability, to find a solution to this system. In fact, the isomorphism between max-plus and min-plus algebra means that the same is true if we replace the ``maximum" with the ``minimum" operator (and set $A(i,j) = +\infty$ when $i$ and $j$ do not interact) \cite{Gun1994}.

In this example, the minimum operator corresponds to the dual type of synchronization: when only the \textit{fastest} machine(s) contribute towards the completion of each product.  In this paper, we consider the intermediate case: what if completing only a proportion $\omega$ of machines was required for making a product?  In this case, we will write the system equations in a similar form as \eqref{equ:maxpluslinear}:
\begin{equation}\label{equ:maxminomega}
    A\otimes_\omega x=b
\end{equation}
noting the new operator $\otimes_\omega$ to indicate this new synchronization protocol, whereby $0<\omega \leq 1$ such that $\ceils{\omega n} = 1$ corresponds to the min-plus system and $\ceils{\omega n} = n$ yields the original max-plus system.  It turns out that such a system as \eqref{equ:maxminomega} is a mixture of maximum and minimum operators, that is, a system of ``max-min-plus" equations \cite{EPthesis}.  We therefore call it the ``maxmin-$\omega$ system".

It is well-known how to solve a max-plus system \eqref{equ:maxpluslinear}, see, e.g., \cite{Butkovic2003} for a comprehensive presentation.  We use an approach similar to \cite{Butkovic2003}, thereby extending those methods to a generalized linear max-min-plus system, specifically the maxmin-$\omega$ system.

For further context, the maxmin-$\omega$ system was introduced as a model of a network of processors such as those described above, where each processor requires only a fraction $\omega$ of machines to produce their component.  When seen as a network, whose nodes adopt some state, maxmin-$\omega$ requires that each node has knowledge of a fraction $\omega$ of neighborhood states before updating its own state  \cite{EPthesis}.  Moreover, knowledge of the neighborhood states is not synchronous; $A(i,j)\in\R$ is the time it takes for nodal state $j$ to be sent to node $i$.  Thus, $A\in\R^{n\times n}$ denotes the matrix of such transmission times associated to a network of $n$ nodes; in graph theoretical terms, this matrix is a weighted adjacency matrix.

By construction, $\omega>0$; otherwise, there is no transmission of information, that is, nothing happens.  We will, therefore, write ``$\omega\approx 0$" to represent the smallest value that $\omega$ can take.  When $\omega\approx 0$, we have a min-plus system, whilst $\omega=1$ yields max-plus dynamics; for $\omega\in(0,1)$, the system is a mixture, i.e., a max-min-plus system.  The next section will introduce max-plus and min-plus algebra more formally but, suffice it to say that, for the motivating application thus far discussed, a min-plus system is one whose nodes wait for the first neighborhood input before updating, whilst a max-plus system requires all neighborhood inputs to arrive before a node updates its state. As already alluded to, the literature on such systems has focused mainly on the two (min-plus and max-plus) extremes.  Applications are not as forthcoming for the intermediate case, but this should not preclude the natural step towards addressing the gap.  Thus, we proceed to consider the case $0<\omega<1$ as both a purely mathematical exercise as well as in anticipation of applications to this effect.  

The rest of this paper is structured as follows. Section 2 presents the basic descriptions on max-plus and min-plus algebra and also maxmin-$\omega$ operation. The problem formulation and the preliminary results are also described in this section. Section 3 discusses the problem when each column of the matrix is distinct while Section 4 addresses the general case by considering the existence of duplicates in each column. The method presented in Sections 3-4 is to find the so-called ``fully active'' solution. Then, in Section 5, we discuss the approach to find other solutions by applying ``relaxation'' i.e., increasing or decreasing some components of each fully active solution. This section also presents the discussion related to the number of solutions. Unlike the classical linear equations in max-plus or min-plus algebra, it is possible that the number of solutions is not unique but finite.
Finally, we present the concluding remarks and the direction for future work in Section 6.






\section{Preliminaries}
This section presents the notations, definitions and some basic results which will be used in the next sections.

Throughout this paper, we always assume that $m,n\geq 1$ are integers and define $M:=\{1,\ldots,m\}$ and $N:=\{1,\ldots,n\}$. The tuple $(a_1,a_2,\ldots, a_n)$ with $n$ elements is a \textit{permutation} of $N$ if $\{a_1,a_2,\ldots,a_n\}=N$. For each multiset (i.e., a set with possible repetitions of elements) $S$, the notation $|S|$ refers to the cardinality (i.e., the number of elements) of $S$. For $0\leq k\leq |S|$, we define $\mathcal{P}(S,k)$ as the set of all $k$-(multi)subsets of $S$. 

The notations $\varepsilon$ and $\Rmax$ stand for $-\infty$ and $\R \cup\{\varepsilon\}$, respectively. For each $a,b\in \Rmax$, we set
\begin{equation}
a\oplus b = \max\{a,b\}~\text{and}~a\otimes b:=a+b.
\label{eq:max-plus-operations}
\end{equation}
The so-called \textit{max-plus algebra} is a semiring $(\Rmax,\oplus,\otimes)$ with $\varepsilon$ and $0$ as the zero and unit elements, respectively \cite{baccelli92}. We denote $\Rmax^{m\times n}$ as the set of $m\times n$ matrices over max-plus algebra. Similarly, $\Rmax^m$ refers to the set of vectors with $m$ elements in max-plus algebra. 

For a matrix $A$, the notation $A(i,j)$ represents the entry of matrix $A$ at $i$-th row and $j $-th column. Furthermore, $A(i,\cdot)$ and $A(\cdot,j)$ are the $i$-th row and $j$-th column of $A$, respectively. The operation \eqref{eq:max-plus-operations} can be extended to matrices and vectors as in conventional linear algebra. For $ A,B\in \Rmax^{m\times n}, C\in \Rmax{n\times p}$ and $\alpha\in\Rmax$, 
\begin{align*}
\nonumber [A\oplus B](i,j) &= A(i,j)\oplus B(i,j)=\max\{A(i,j),B(i,j)\},\\
\nonumber [A\otimes C](i,j) &= \bigoplus_{k=1}^n A(i,k)\otimes C(k,j)=\max_{1\leq k\leq n}\{A(i,k)+C(k,j)\},\\
\nonumber [\alpha\otimes A](i,j) &= \alpha\otimes A(i,j)=\alpha+A(i,j).
\end{align*}

Min-plus algebra, the dual of max-plus algebra, is a semiring $(\Rmin,\oplus^\prime, \otimes^\prime)$ where $\Rmin:=\R\cup\{\xi:=+\infty\}$,
\begin{equation}
    \label{eq:minplus_operations}
    a\oplus^\prime b:=\min\{a,b\}~\text{and}~a\otimes^\prime b=a+b.
\end{equation}
The notations and operations with matrices in min-plus algebra are defined similarly to those of max-plus algebra. Observe that $a\oplus^\prime b=-((-a)\oplus (-b))$.

Given a multiset of real numbers $S=\{s_1,s_2,\ldots,s_n\}$, the \textit{maxmin-$\omega$} operation 
\begin{equation}
\label{eq:omega-operation}
    \bomega S
\end{equation}
yields the $\ceils{\omega n}$-th smallest element of $S$ for $0<\omega \leq 1$. One can say that \eqref{eq:omega-operation}  is the generalization for $\oplus$ and $\oplus^\prime$ operations. It is straightforward to see that when $\ceils{\omega n}=1$ $(\omega\approx 0)$, operation \eqref{eq:omega-operation} corresponds to min-plus addition. On the other hand, if $\lceil \omega n \rceil=n$ $(\omega=1)$ then \eqref{eq:omega-operation} is a max-plus addition. 

The following proposition shows that \eqref{eq:omega-operation} can be expressed as the combination of $\oplus$ and $\oplus^\prime$ operations.

\begin{prop}[\cite{EPthesis}]
Given a multiset $S\subseteq \mathbb{R} $ with $|S|=n$ and suppose $p=\lceil \omega n\rceil$, then
\begin{align}
\label{eq:bomega-min}
    \bomega S &= \bigoplus_{P\in \mathcal{P}(S,p)}\hspace*{-2ex}{}^\prime\left\{ \bigoplus P\right\}\\
    \label{eq:bomega-max}
    &= \bigoplus_{P\in \mathcal{P}(S,n+1-p)}\left\{ \bigoplus{}^\prime P\right\}
\end{align}
\end{prop}

Properties~\eqref{eq:bomega-min} and~\eqref{eq:bomega-max} show that the operator $A\otimes_{\omega}$ is a min-max function, i.e., it belongs to the class of functions considered in \cite{GG-98,Gun1994,Subiono,SO-97}.
\subsection{Linear Equations for Maxmin-$\omega$ Systems}
Given a matrix $A\in \R^{m\times n}$ and $b\in\R^m$, the corresponding linear equation for the maxmin-$\omega$ systems is a problem to find vectors $x\in \R^n$ such that
\begin{equation}
\label{eq:problem}
A\otimes_\omega x=b.
\end{equation}
\begin{defn}
\normalfont
\label{defn:normalized}
The linear equation \eqref{eq:problem} is called \textit{normalized} if $b=\textbf{0}$, where $\textbf{0}$ is a vector whose all elements are 0.
\end{defn}

Any linear equation \eqref{eq:problem} can be normalized by defining a matrix $A^\ast$ where
\[
A^\ast(i,\cdot) = -b_i\otimes A(i,\cdot).
\]
It is easy to see that $A\otimes_\omega x=b$ if and only if $A^\ast\otimes_\omega x=\textbf{0}$. The following proposition asserts a trivial condition in which a normalized linear equation \eqref{eq:problem} is not solvable.
\begin{prop}
Given a normalized linear equation \eqref{eq:problem}. If there are $i,j\in M$ such that $A(i,\cdot)>A(j,\cdot)$, then \eqref{eq:problem} is not solvable for any $\omega$. 
\end{prop}
\begin{defn}
\normalfont
\label{defn:col_distinct}
Suppose we have a normalized linear equation \eqref{eq:problem}. If each column of $A$ has distinct elements, then we call it \textit{column distinct}; otherwise it is \textit{column indistinct}.
\end{defn}

We define
\begin{align}
\label{eq:solution-set}
\mathcal{S}(A,\omega) &= \{x\in \mathbb{R}^n\mid A\otimes_\omega x=\textbf{0}\}
\end{align}
as the set of solutions for normalized system \eqref{eq:problem}. If $x=[x_1\ldots x_n]^\top \in \mathcal{S}(A,\omega)$, then for each $i\in M$ there exists at least one $j\in N$ such that $A(i,j)+x_j=0$. The corresponding $A(i,j)$ is called as an ``active element'' of $A$ w.r.t. $x$ and $\omega$. We call a solution $x$ of \eqref{eq:problem} as ``fully active'' if there is at least one active element in each column of $A$. 

We define a set of tuples corresponding to the active elements as follows
\begin{align*}
    \mathcal{I}(A,\omega,x) &=\{(i_1,i_2,\ldots,i_n)\mid A(i_k,k)+x_{k}=0 ~\text{for all}~k\in N\}.
\end{align*}
It should be noted that we may have $|\mathcal{I}(A,\omega,x)|>1$ because, for some $i_k\in M$, we may have many distinct $k$ such that $A(i_k,k)+x_{k}=0$. It is also possible that $|\mathcal{I}(A,\omega,x)|=0$; that is, only when $x$ is not a fully active solution.
Finally, we define the set of ``solution indices'' for \eqref{eq:problem} as follows
\begin{align}
    \label{eq:solution-index}
    \mathcal{I}(A,\omega) = \bigcup_{x\in \mathcal{S}(A,\omega)} \mathcal{I}(A,\omega,x).
\end{align}
It is straightforward to see that $|\idxset|\leq |\solset|$. Furthermore, if $\solset$ is not empty then so is $\idxset$.

The linear equation problem has been well studied for max-plus and min-plus setups; that is, when $\omega\approx 0$ or $\omega =1$ in \eqref{eq:problem}. As discussed in \cite{Butkovic2010}, the solvability of the problem can be determined from a \textit{principal (candidate) solution} $\bar{x}=[\bar{x}_1\cdots \bar{x}_n]^\top$ defined\footnote{In \cite{Butkovic2010}, only the max-plus case is formally considered. However, similar results can be obtained in the min-plus case due to the isomorphism between the max-plus and min-plus semirings.} by
\begin{equation}
    \label{eq:principal-solution}
\bar{x}_k = -\bomega A(\cdot,k)
\end{equation}
and the sets $M_1,\ldots,M_n$ defined by
\begin{equation}
    \label{eq:principal-sets}
    M_k = \left\{ i\in M \mid A(i,k)=\bomega A(\cdot,k)\right\}.
\end{equation}
\begin{prop}[\cite{Butkovic2010}]
Given a normalized linear equation \eqref{eq:problem} for $\omega\approx 0$ or $\omega =1$.  Suppose $\bar{{x}}$ is defined by \eqref{eq:principal-solution} and $M_1,\ldots,M_n$ defined by \eqref{eq:principal-sets}. The following statements are equivalent
\begin{itemize}
    \item[{\rm (a)}] $\mathcal{S}(A,\omega)\neq \emptyset$,
    \item[{\rm (b)}] $\bar{x}\in \mathcal{S}(A,\omega)$,
    \item[{\rm (c)}] $\displaystyle \bigcup_{k\in N} M_k=M$.
\end{itemize}
\end{prop}

\begin{prop}[\cite{Butkovic2010}]
\label{prop:uniquenes-butkovic}
Given a normalized linear equation \eqref{eq:problem} for $\omega\approx 0$ or $\omega =1$. Suppose $\bar{x}$ is defined by \eqref{eq:principal-solution} and $M_1,\ldots,M_n$ defined by \eqref{eq:principal-sets}. $\mathcal{S}(A,\omega)=\bar{x}$ if and only if
\begin{itemize}
    \item[{\rm (i)}] $\displaystyle \bigcup_{k\in N} M_j=M$.
    \item[{\rm (ii)}] $\displaystyle \bigcup_{k\in N^\prime} M_j\neq M$ for any $N^\prime\subseteq N,N^\prime\neq N$.
\end{itemize}
\end{prop}

In Sections \ref{sec:colum-distinct}-\ref{sec:general-case}, we will discuss the methods to solve \eqref{eq:problem} for any $\omega \in (0,1]$ in which the matrices have only finite elements i.e., $A\in\R^{m\times n}$. As expected, the resulting technique coincides with the ones discussed in \cite{Butkovic2010} for $\omega\approx 0$ or $\omega =1$. 

\section{Column Distinct Cases}
\label{sec:colum-distinct}

This section discusses the strategies to solve \eqref{eq:problem} when $A\in \R^{m\times n}$ is a column distinct matrix. Instead of directly trying to find $x\in \solset$, our methods will look for the tuple $(i_1,\ldots,i_n)\in \idxset$. Notice that, such a tuple corresponds to a fully active solution. 
Propositions \ref{prop:modif-matrix}-\ref{prop:candidate-solutions} provide the conditions for such tuples while Proposition~\ref{prop:fully active-idx} shows the uniqueness of solution index for each fully active solution.
\begin{prop}
\label{prop:fully active-idx}
    Given a normalized \eqref{eq:problem} where $A\in \R^{m\times n}$ is a column distinct matrix. If $x\in \solset$ is a fully active, then $|\mathcal{I}(A,\omega,x)|=1$. 
\end{prop}
\begin{proof}
Let us define a matrix $C$ where $C(i,j)=A(i,j)+x_j$ for $i\in M$ and $j\in N$. Notice that, the active elements of $A$ w.r.t. $x$ and $\omega$ corresponds to the zeros of $C$. Since $A$ is column distinct, there is exactly one zero element in each column of $C$. Hence, $|\mathcal{I}(A,\omega,x)|=1$.
\end{proof}




\begin{prop} 
\label{prop:modif-matrix}
Consider a normalized system \eqref{eq:problem} where $A\in \R^{n\times n}]$ is a column distinct matrix and suppose that a matrix $B$ is generated by modifying a single element of $A$ $($say $A(k,l))$. If the ordering of elements for $A(\cdot,l)$ and $B(\cdot, l)$ is the same, then $\idxset(A)=\idxset(B)$ for each $\omega\in (0,1]$.
\end{prop}
\begin{proof}
We only provide the proof for $\idxset(A)\subseteq \idxset(B)$. The proof for the relation $\idxset[B]\subseteq \idxset(A)$ can be written similarly. 

Suppose that $(i_1,\ldots,i_n)\in \idxset$. Then, there exists a vector $x=[x_1~\cdots~x_n]^\top\in \solset$ such that $x_{j}=-A(i_j,j)$ for $j\in N$. Let us define two matrices $C_1,C_2$ where $C_1(i,j)=A(i,j)+x_j$ and $C_2(i,j)=B(i,j)+x_j$ for $i\in M$ and $j\in N$. Indeed, observe that the ordering of $C_1(\cdot,l)$ and $C_2(\cdot,l)$ is also the same. Notice that for each $i\in N$ there is at least one zero element in $C_1(i,\cdot)$. Furthermore, there are at most $p-1$ negative values and at most $n-p$ positive values in $C_1(i,\cdot)$ for any $i$, where $p=\lceil \omega n\rceil$. Since matrices $A$ and $B$ only differ at the $k$-th row, we only need to consider $C_1(k,\cdot)$ and $C_2(k,\cdot)$. We have two possible cases: $i_l\neq k$ and $i_l=k$.

The first case happens when $A(k,l)$ is not active element. Furthermore, $C_1(k,l)$ and $C_2(k,l)$ have the same sign (either both positive or both negative), since the ordering of $C_1(\cdot,l)$ and $C_2(\cdot,l)$ is the same and $0$ occurs in both columns in the same positions. Consequently, the number of positive (resp. negative) elements at $C_2(k,\cdot)$ is at most $p-1$ (resp. $n-p$), respectively. Thus, $x\in \solset(B)$ and $(i_1,\ldots,i_n)\in \idxset(B)$. For the second case, we have $C_1(i_l,l)=0$ but $C_2(i_l,l)\neq 0$. In this case, it can be seen that the vector $y=[y_1 ~\cdots~y_n]$ where
\begin{align*}
y_{j}=\left\{
    \begin{array}{ll}
        x_{j} & ~\text{if}~j\neq l, \\
         -B(i_l,l)&~\text{if}~j=l,
    \end{array}\right.
\end{align*}
satisfies $B\otimes_\omega y=\textbf{0}$ and $(i_1,\ldots,i_n)\in \idxset(B)$. 
\end{proof}

Assuming $A\in \R^{m\times n}$, we define matrix $\bar A$ as follows: for each $j\in N$
\begin{equation}
\label{principal}
    \bar A(i,j) = k~\text{if}~\bigoplus\hspace*{-0.5ex} {}_{\frac{k}{m}} A(\cdot,j)=A(i,j),
\end{equation}
where $k\in M$. For further reference, we call $\bar A$ the \textit{principal order matrix} associated with $A$. Notice that, $\bar{A}(i,j)=k$ implies that $A(i,j)$ is the $k$-smallest element of $A(\cdot,j)$. Furthermore, since the elements of $A(\cdot,j)$ are distinct we have \[\{\bar A(1,j),\bar A(2,j),\ldots,\bar A(n,j)\}=N.\]

We observe that replacing $A$ by the associated principal order matrix does not change the set of solution indices for \eqref{eq:problem}. Furthermore, it allows to add an additional necessary condition for such solution indices.
\begin{prop}
\label{prop:sol-idx-abar}
For a normalized \eqref{eq:problem} where $A\in \R^{m\times n}$ is a column distinct matrix, we have $\mathcal{I}(A,\omega)=\mathcal{I}(\bar{A},\omega)$ for all $\omega \in (0,1]$.
\end{prop}
\begin{proof}
Notice that, the order of the elements at $A(\cdot,l)$ is the same as that of $\bar{A}(\cdot,l)$ for $l\in N$. Moreover, $\bar A$ can be generated first by adding a big enough constant to $A$ and then by lowering of entries in each column in the order of their ascendance, which can be interpreted as the `composition' of modifications (at most $n^2$ times) of matrix $A$ mentioned in Proposition~\ref{prop:modif-matrix}. As a result, the set of solution indices remains the same.
\end{proof}
\begin{prop}
\label{prop:candidate-solutions}
For a normalized \eqref{eq:problem} where $A\in \R^{m\times n}$ is a column distinct matrix, if a tuple $(i_1,\ldots,i_n)\in \idxset(\bar{A})$ then the following conditions hold
\begin{itemize}
    \item[{\rm (i)}] $\{i_1,i_2,\ldots,i_n\}=M$,
    \item[{\rm (ii)}] $\displaystyle mp \leq \sum_{k\in N} \bar{A}(i_k,k)\leq mp+n-m$,
where $p=\lceil\omega n\rceil$.
\end{itemize}
\end{prop}
\begin{proof} Suppose $(i_1,\ldots,i_n)\in \idxset(\bar{A})$ and $x=[x_1~\cdots~x_n]^\top\in \solset(\bar{A})$ is the corresponding fully active solution such that $x_{k}=-\bar{A}(i_k,k)$ for $k\in N$.  (i) The first condition follows from the fact that, for each $i\in M$ there must be at least one active elements w.r.t $x$ and $\omega$ in $\bar{A}(i,\cdot)$.\\
(ii) Let us define a matrix $C$ where $C(i,j)=\bar A(i,j)+x_j$ for $i\in M,j\in N$. Notice that, there are at most $p-1$ negative elements in each row of $C$. Furthermore, there are exactly $-x_{k}-1=\bar A(i_k,k)-1$ negative elements in the $k$-th column of $C$. Hence, we obtain
\[
\sum_{k\in N} (\bar{A}(i_k,k)-1)\leq m(p-1).
\]

On the other hand, there are at most $n-p$ positive element in each row $C$ and exactly $m+x_k=m-\bar{A}(i_k,k)$ positive elements in the $k$-column of $C$. These conditions implies that 
\[
\sum_{k\in N} (m-\bar{A}(i_k,k))\leq m(n-p).
\]
By simple algebraic manipulations with the above inequalities, one can obtain the desired lower and upper bounds for $\sum_{k\in N} \bar{A}(i_k,k)$. 



\end{proof}

Proposition~\ref{prop:candidate-solutions} is important since it gives a number of candidates for solution indices. We further provide the conditions for such indices when $m>n$ and $m=n$.

\begin{cor}
\label{cor:m>n-cases}
    Suppose we have a normalized \eqref{eq:problem} where $A\in \R^{m\times n}$ is a column distinct matrix. If $m>n$ then $\idxset=\solset=\emptyset$.
\end{cor}
\begin{proof}
    The non-solvable condition when $m>n$ is due to the fact the lower and upper bounds for $ \sum_{k\in N} \bar{A}(i_k,k)$ given by Proposition~\ref{prop:candidate-solutions} are not consistent.
\end{proof}

\begin{cor}
\label{cor:m=n-cases}
    Suppose we have a normalized \eqref{eq:problem} where $A\in \R^{m\times n}$ is a column distinct matrix. If $m=n$ then the following conditions hold 
    \begin{itemize}
        \item[{\rm (i)}] if $(i_1,\ldots,i_n)\in \idxset$ then $ \sum_{k\in N} \bar{A}(i_k,k)=\lceil\omega n\rceil n$,
        \item[{\rm (ii)}] each $(i_1,\ldots,i_n)\in \idxset$ is a permutation of $N$,
        \item [{\rm (iii)}] all solutions ${x}\in \solset$ are fully active,
        \item[{\rm (iv)}] $|\solset|=|\idxset|$ for each $\omega\in(0,1]$.
    \end{itemize}
\end{cor}
\begin{proof} 
\begin{itemize}
    \item[(i)] Direct result of Proposition~\ref{prop:candidate-solutions} when $m=n$.
    \item[(ii)] Suppose that $(i_1,\ldots,i_n)\in \idxset$ and $x\in \solset$ such that $x_k=-A(i_k,k)$ for each $k\in N$. Since $A$ is column distinct and $m=n$, there is exactly one active element w.r.t. $x$ in each row of $A$. Consequently, $i_1,\ldots,i_n$ must be distinct.
    \item[(iii)] Suppose that $x\in \solset$. Since $A$ is column distinct, in each column of $A$ there is at most one active element w.r.t. $x$. The condition $m=n$ enforces that all columns of $A$ to contain one active element.
    \item[(iv)] By part (ii) and (iii), for each $x\in \solset$ we have $|\mathcal{I}(A,\omega,x)|=1$. This condition implies that there is one-to-one correspondence between $\solset$ and $\idxset$. 
\end{itemize}
\end{proof}

As we are going to see in examples (for square and non-square cases) below, it often happens that \eqref{eq:problem} has several solutions stemming from these ``candidates''. 
\begin{exmp}
    \label{ex:distinct-3x4}
    Suppose we have a normalized linear equation \eqref{eq:problem} where
\[
A=\begin{bmatrix}
    5&5&-2&3\\
    2&4&6&1\\
    6&-1&7&2
\end{bmatrix}
~\text{and}~\omega\in\left\{\frac{1}{4},\frac{1}{2},\frac{3}{4},1\right\}.
\]
We will find the solutions of \eqref{eq:problem} for each $\omega$. The principal order matrix is
\[
\bar{A}=\begin{bmatrix}
2&3&1&3\\
1&2&2&1\\
3&1&3&2
\end{bmatrix}.
\]
By Proposition~\ref{prop:candidate-solutions}, if $(i_1,i_2,i_3,i_4)\in \idxset(\bar{A})$ then $\{i_1,i_2,i_3,i_4\}=\{1,2,3\}$ and 
\begin{equation}
\label{e:omega12}
12\omega\leq \bar{A}(i_1,1)+\ldots+ \bar{A}(i_4,4)\leq 12\omega+1.
\end{equation}
\begin{itemize}
    \item[(i)] For $\omega=\frac{1}{4}$, the only tuple satisfying~\eqref{e:omega12} is $(2,3,1,2)$ which corresponds to a vector $\bar{x}=[-1~-1~-1~-1]^\top$. One could check that $\bar{A}\otimes \bar{x}=\emph{\textbf{0}}$. Consequently, we have $\idxset(A,\frac{1}{4})=\idxset(\bar{A},\frac{1}{4})=\{(2,3,1,2)\}$ 
    which implies that $\begin{bmatrix}
    -2&1&2&-1
    \end{bmatrix}^\top \in \solset(A,\frac{1}{4})$.
     \item[(ii)] For $\omega=\frac{1}{2}$, there are 9 tuples satisfying~\eqref{e:omega12}, namely, 
     \[
     \begin{array}{l}
        (1,2,1,3),(1,3,2,2),(1,3,2,3),(2,1,1,3),(2,2,1,3),\\
        (2,3,1,1),(2,3,2,1),(3,2,1,2),(3,3,1,2),
     \end{array}
     \]
      which respectively corresponds a vector
      \[
      \bar{x}\in \left\{
      \begin{bmatrix}
          -2\\-2\\-1\\-2
      \end{bmatrix}\!\!,
      \begin{bmatrix}
          -2\\-1\\-2\\-1
      \end{bmatrix}\!\!,
      \begin{bmatrix}
          -2\\-1\\-2\\-2
      \end{bmatrix}\!\!,
      \begin{bmatrix}
          -1\\-3\\-1\\-2
      \end{bmatrix}\!\!,
      \begin{bmatrix}
          -1\\-2\\-1\\-2
      \end{bmatrix}\!\!,
      \begin{bmatrix}
          -1\\-1\\-1\\-3
      \end{bmatrix}\!\!,
      \begin{bmatrix}
          -1\\-1\\-2\\-3
      \end{bmatrix}\!\!,
      \begin{bmatrix}
          -3\\-2\\-1\\-1
      \end{bmatrix}\!\!,
            \begin{bmatrix}
          -3\\-1\\-1\\-1
      \end{bmatrix}
      \right\}\!.
      \]
      One could check that only the last four vectors satisfy $\bar{A}\otimes_{\frac{1}{2}} \bar{x}=\emph{\textbf{0}}$. Hence, from the same tuples, one can obtain the set of fully active solutions $x\in \solset(A,\frac{1}{2})$, namely,
      \[
    \begin{bmatrix}
        -2\\1\\2\\-3
    \end{bmatrix},
    \begin{bmatrix}
        -2\\1\\-6\\-3
    \end{bmatrix},
    \begin{bmatrix}
        -6\\-4\\2\\-1
    \end{bmatrix},
    \begin{bmatrix}
        -6\\1\\2\\-1
    \end{bmatrix}
    \!\!.
      \]
        \item[(iii)] For $\omega=\frac{3}{4}$, there are 12 tuples satisfying~\eqref{e:omega12}, namely 
     \[
     \begin{array}{l}
         (1,1,2,3),(1,1,3,2),(1,2,3,1),(1,2,3,3),(2,1,3,1),(2,2,3,1)\\
         (3,1,2,2),(3,1,2,3),(3,1,3,2),(3,2,1,1),(3,2,2,1),(3,3,2,1).\\
     \end{array}
     \]
     However, only the last three vectors which belongs to $\idxset(\bar A,\frac{3}{4})$. From the same tuples, one can obtain the set of fully active solutions $x\in \solset(A,\frac{3}{4})$, namely
     \[
     \begin{bmatrix}
         -6\\-4\\2\\-3
     \end{bmatrix},
     \begin{bmatrix}
         -6\\-4\\-6\\-3
     \end{bmatrix},
     \begin{bmatrix}
         -6\\1\\-6\\-3
     \end{bmatrix}.
     \]

      \item[(iv)] For $\omega=1$, there is no tuple that satisfies~\eqref{e:omega12}. Hence, $\idxset(\bar{A},1)=\idxset(A,1)=\solset(A,1)=\emptyset$.
\end{itemize}
\end{exmp}
\begin{exmp}
\label{ex:distinct-3x3}
Suppose we have a normalized linear equation \eqref{eq:problem} where
\[
A=\begin{bmatrix}
    4&7&2\\
    5&2&5\\
    8&3&1
\end{bmatrix}~\text{and}~\omega\in\left\{\frac{1}{3},\frac{2}{3},1\right\}.
\]
We will find the solutions of \eqref{eq:problem} for each $\omega$. The principal order matrix is
\[
\bar{A}=\begin{bmatrix}
1&3&2\\
2&1&3\\
3&2&1
\end{bmatrix}.
\]
\begin{itemize}
    \item[(i)] For $\omega=\frac{1}{3}$, the only candidate for a solution index is $(1,2,3)$ which corresponds to a vector $\bar{x}=\begin{bmatrix}-1&-1&-1\end{bmatrix}^\top$. One could check that $\bar{A}\otimes_{\frac{1}{3}} \bar{x}=\emph{\textbf{0}}$. Consequently, we have $\idxset(A,\frac{1}{3})=\idxset(\bar{A},\frac{1}{3})=\{(1,2,3)\}$ which implies $\solset(A,\frac{1}{3})=\{\begin{bmatrix}
    -4&-2&-1
    \end{bmatrix}^\top\} $.
    \item[(ii)] For $\omega=1$, the only candidate for a solution index is $(3,1,2)$ which corresponds to a vector $\bar{x}=\begin{bmatrix}-3&-3&-3\end{bmatrix}^\top$. One could check that $\bar{A}\otimes \bar{x}=\emph{\textbf{0}}$. Consequently, we have $\idxset(A,1)=\idxset(\bar{A},1)=\{(3,1,2)\}$ which implies $\solset(A,1)=\{\begin{bmatrix}
    -8&-7&-5
    \end{bmatrix}^\top\} $.
    \item[(iii)] For $\omega=\frac{2}{3}$, we have four candidates of solution indices, i.e., $(1,3,2)$, $(2,1,3),(2,3,1),(3,2,1)$. They correspond to vectors
    \[
    \bar{x} \in\left\{
    \begin{bmatrix}
    -1\\-2\\-3
    \end{bmatrix},
    \begin{bmatrix}
    -2\\-2\\-2
    \end{bmatrix},
    \begin{bmatrix}
    -2\\-3\\-1
    \end{bmatrix},
    \begin{bmatrix}
    -3\\-1\\-2
    \end{bmatrix}
    \right\}.
    \]
    One can check that all above vectors satisfy $\bar A\otimes_{\frac{2}{3}}\bar{x}=\emph{\textbf{0}}$. Hence, we have $\idxset(A,\frac{2}{3})=\idxset(\bar{A},\frac{2}{3})=\{(1,3,2), (2,1,3),(3,1,2),(3,2,1)\}$ and
    \begin{center}
    $\solset(A,\frac{2}{3})=\left\{
    \begin{bmatrix}
    -4\\-3\\-5
    \end{bmatrix},
    \begin{bmatrix}
    -5\\-3\\-2
    \end{bmatrix},
    \begin{bmatrix}
    -5\\-7\\-1
    \end{bmatrix},
    \begin{bmatrix}
    -8\\-2\\-2
    \end{bmatrix}\right
    \}.$    
    \end{center}
    
\end{itemize}
The linear equation $A\otimes_{\omega}x=\emph{\textbf{0}}$ is thus solvable for all possible $\omega$. Note that for $\omega=\frac{2}{3}$, the number of solutions is more than 1 but finite.
\end{exmp}
\begin{rem}
{\rm 
The previous example shows that $A\otimes_{\omega}x={\textbf{0}}$
can have a finite number of different solutions. This shows that for general $\omega$ the solution set to this equation can be disconnected in the topological sense. This is in contrast with solution sets of $A\otimes x=b$ in max-plus or min-plus algebras, which are not only connected but also tropically convex (with respect to the convexities induced by max-plus and min-plus segments, respectively).
}
\end{rem}

\begin{cor}
\label{cor:solution}
Consider a normalized system \eqref{eq:problem} where $A\in \R^{n\times n}$ is a column distinct matrix. If $\ceils{\omega n}\in \{1,n\}$, then $|\solset|\in \{0,1\}$. Moreover, if $n=3$ and $|\solset|=1$ for $\ceils{\omega n}\in \{1,3\}$, then $|\solset|=4$ for $\ceils{\omega n}=2$.
\end{cor}
\begin{proof}
    By Proposition~\ref{prop:candidate-solutions},
    if $(i_1,\ldots,i_n)\in \idxset(\bar{A})$ then $\bar A(i_1,1)+\ldots+\bar A(i_n,n)=\ceils{\omega n}n$ .
    For $\ceils{\omega n}\in \{1,n\}$, there is only one possibility for such tuple i.e., when $\bar A(i_k,k)=\ceils{\omega n}$ for $k\in N$. Furthermore, Corollary~\ref{cor:m=n-cases}(ii) asserts that $|\idxset(\bar{A})|=1$ if and only if $(i_1,\ldots,i_n)$ is a permutation of $N$.
    Hence, $|\idxset(\bar{A})|\in\{0,1\}$. Finally, Proposition~\ref{prop:sol-idx-abar} and Corollary~\ref{cor:m=n-cases}(iv) imply that $|\idxset(\bar{A})|=|\idxset|=|\solset|$, which completes the proof.

    Notice that, for $n=3$ and when $|\solset|=1$ for $\ceils{3\omega}\in \{1,3\}$, the elements at each row of $\bar{A}$ are distinct. 
    Corollary~\ref{ex:distinct-3x3} demonstrates that, for such $\bar{A}$, $|\idxset(\bar{A})|=|\idxset|=|\solset|=4$ when $\ceils{3\omega}=2$. Notice that, if one permute the rows or columns of $\bar{A}$, the number of solutions for each $\omega$ remains the same. 
    
    
\end{proof}

\section{General Case}
\label{sec:general-case}
This section presents the approach to solve \eqref{eq:problem} when $A$ is not necessarily column distinct i.e., there are may be duplicates on several columns of $A$. 
Inspired by the method described in Section \ref{sec:colum-distinct}, we develop a technique to solve \eqref{eq:problem} by first finding a tuple $(i_1,\ldots,i_n)$ which belongs to $\idxset$. 

 To allow for the the possibility of equal entries in the columns, the generation of the principal order matrix $\bar{A}$ has to be slightly different. Suppose now $\texttt{col}_j$ is the set of elements at the $j$-th column of $A$ after removing the duplicates\footnote{For instance, if $A(\cdot,j)=[4~5~5]^\top$ then $\texttt{col}_j=\{4,5\}$.}. Furthermore, for a matrix $A$ and $r\in \mathbb{R}$, we define $\texttt{idx}_j(A,r)$ as the set of indices where $r$ appears at $A(\cdot,j)$. In other words,
\begin{equation}
    \texttt{idx}_j(A,r)=\{i\in M\mid A(i,j)=r\}.
\end{equation}
The principal order matrix from a non-column distinct matrix $\bar{A}$ is generated as follows
\begin{equation}
\label{principal-indistinct}
    \bar{A}(i,j) = 1+ \sum_{l=1}^{k-1} |\texttt{idx}_j(A,\bigoplus\hspace*{-0.5ex} {}_{\frac{l}{|\texttt{col}_j|}} \texttt{col}_j)|~\text{if}~\bigoplus\hspace*{-0.5ex} {}_{\frac{k}{|\texttt{col}_j|}} \texttt{col}_j=A(i,j),
\end{equation}
for $1\leq k\leq |\texttt{col}_j|,i\in M$ and $j\in N$. Intuitively, if $A(i,j)$ is the $k$-th smallest element at $\texttt{col}_j$, then $\bar{A}(i,j)-1$ corresponds to the sum of the number of appearances for other elements in $\texttt{col}_j$ which are smaller than $A(i,j)$. Notice that, since $A$ is not column distinct in general, neither is $\bar{A}$. Suppose now $f_j$ denotes the largest number of appearances of elements at $\bar{A}(\cdot,j) $; that is
\begin{equation}
    \label{eq:largest-freq}
    f_j=\max\{|\texttt{idx}_j(\bar{A},1)|,|\texttt{idx}_j(\bar{A},2)|,\ldots, |\texttt{idx}_j(\bar{A},m)|\}.
\end{equation}
It is evident that $f_1+\ldots+f_n\geq n$. The following proposition presents a necessary condition for a tuple $(i_1,\ldots,i_n)$ to belong to $\idxset(\bar A)$.

\begin{prop}
\label{prop:candidate-solutions-indistinct}
    For a normalized system~\eqref{eq:problem} with $A\in \R^{m\times n}$, if  $(i_1,\ldots,i_n)\in \idxset(\bar{A})$ then the following conditions hold:
    \begin{itemize}
    \item[{\rm (i)}] $\emph{\texttt{idx}}_1(\bar{A},\bar{A}(i_1,1))\cup \ldots \cup \emph{\texttt{idx}}_n(\bar{A},\bar{A}(i_n,n))=M$,
    \item[{\rm (ii)}] $mp + n - f \leq \bar{A}(i_1,1)+\ldots+\bar{A}(i_n,n)\leq  mp+n-m$,
where $p=\ceils{\omega n}$ and $ f= f_1+\ldots+ f_n$.
\end{itemize}    
\end{prop}
\begin{proof} (i) The proof for the first part can be done similarly as in Proposition~\ref{prop:candidate-solutions}.\\
    (ii) Suppose that $(i_1,\ldots,i_n)\in \idxset(\bar{A})$ and  $x=[x_1~\cdots~x_n]^\top\in \solset(\bar{A})$ such that $x_{k}=-\bar{A}(i_k,k)$ for $k\in N$. Let us define a matrix $C$ where $C(i,j)=\bar A(i,j)+x_j$ for $i\in M$ and $j\in N$. Notice that, the number of zeros at $C$ is equal to ${d}_{1}(\bar A,\bar A(i_1,1))+\ldots+{d}_{n}(\bar A,\bar A(i_n,n))$. Hence, several columns and rows of $C$ may have multiple zeros. Since $\bigoplus_{\omega} C(i,\cdot)=0$ for all $i\in M$, there are at most $p-1$ negative elements and at most $n-p$ positive elements at $C(i,\cdot)$. In total, there are at most $m(p-1)$ negative elements and $m(n-p)$ positive elements at $C$.

    On the other hand, at the $k$-column of $C$, there are exactly $-x_{k}-1=\bar A(i_k,k)-1$ negative elements. Consequently, 
\[
\sum_{k\in N} \bar{A}(i_k,k)-n\leq m(p-1)
\]
which yields $\sum_{k\in N} \bar{A}(i_k,k)\leq mp+n-m$. Similarly, at the $k$-column of $C$, there are exactly \[ m-\bar A(i_k,k)- g_k +1\] positive elements where $g_k$ is the number of appearance of $\bar{A}(i_k,k)$ at $\bar{A}(\cdot,k)$ i.e., $g_k=|\texttt{idx}_k(\bar{A},\bar{A}(i_k,k))|$. Hence, 
\[
\sum_{k\in N} (m-\bar A(i_k,k)-g_k+1)\leq m(n-p)
\]
which yields $\sum_{k\in N} \bar{A}(k,j_k)\geq mp+n-g_k$. Furthermore, since $ g_k\leq f_{k}$ for $k\in N$, we have 
\[\sum_{k\in N} \bar{A}(k,j_k)\geq mp+n-\sum_{k\in N} f_{k}.\]
This completes the proof.
\end{proof}

Unlike the column distinct cases, the condition that the number of rows is greater than that of columns does not necessarily imply that the problem \eqref{eq:problem} is not solvable. Instead, the magnitude of the sum of $f_j$ defined in \eqref{eq:largest-freq} may lead to the unsolvability. Furthermore, unlike in column distinct case, $1\leq |\solset|<+\infty$ may happen even when the matrix is not square.
\begin{cor}
\label{cor:m>f}
Suppose we have a normalized \eqref{eq:problem} with $A\in \R^{m\times n}$. If $m>f_1+\ldots+f_n$ then $\idxset=\solset=\emptyset$.
\end{cor}

\begin{cor}
\label{cor:m=f-cases-indistinct}
    Suppose we have a normalized \eqref{eq:problem} with $A\in \R^{m\times n}$. If $m=f_1+\ldots+f_n$ then the following conditions hold 
    \begin{itemize}
        \item[{\rm (i)}] if $(i_1,\ldots,i_n)\in \idxset$ then $ \sum_{k\in N} \bar{A}(i_k,k)=\ceils{\omega n}m+n-m$; 
        \item[{\rm (ii)}] if $(i_1,\ldots,i_n)\in \idxset$ then  $i_1,\ldots,i_n$ are distinct elements;
        \item [{\rm (iii)}] all solutions $x\in \solset$ are fully active;
        \item[{\rm (iv)}] $|\solset|=|\idxset|$ for each $\omega\in(0,1]$.
    \end{itemize}
\end{cor}
\begin{proof}
    The proofs are similar to those of Corollary~\ref{cor:m=n-cases}. It is important to note that, for a fully active solution $x$, if $m=f_1+\ldots+f_n$ then there is exactly one active element w.r.t $\omega$ and $x$ at each row of $A$. 
\end{proof}

\begin{rem}
One can see that, \eqref{principal-indistinct} is the generalization of \eqref{principal}; that is, if $A$ is distinct-column, then $ |\emph{\texttt{col}}_i|=n$ and $|\emph{\texttt{idx}}_i(A,\bigoplus\hspace*{-0.5ex} {}_{\frac{l}{|\emph{\texttt{col}}_i|}} \emph{\texttt{col}}_i)|=1$ for each $i\in N$. Following this, Proposition~\ref{prop:candidate-solutions-indistinct} is also the generalization of Proposition~\ref{prop:candidate-solutions} due to the fact that if $A$ is column distinct, then $f_1=f_2=\cdots=f_n=1$.
\end{rem}
\begin{exmp}
Suppose we have a normalized linear equation \eqref{eq:problem} where
    \[
    A = \begin{bmatrix}
        1&4&2\\
        1&2&4\\
        3&1&3\\
        4&3&1
    \end{bmatrix} ~\text{and}~ \omega\in \left\{\frac{1}{3},\frac{2}{3},1\right\}.
    \]
    We will find the fully active solutions of \eqref{eq:problem} for each $\omega$. Notice that, $\bar A=A$. Furthermore, $f_1=2,f_2=f_3=1$ and $m=f_1+f_2+f_3=4$. By Proposition~\ref{prop:candidate-solutions}, if $(i_1,i_2,i_3)\in \idxset[(bar{A})$ then \[\emph{\texttt{idx}}_1({A},{A}(i_1,1))\cup \emph{\texttt{idx}}_2(A,A(i_2,2)) \cup \emph{\texttt{idx}}_3({A},{A}(i_3,3))=\{1,2,3,4\}\] and ${A}(i_1,1)+A(i_2,2)+ A(i_3,3)= 12\omega-1$.
    \begin{itemize}
        \item[(i)] For $\omega=\frac{1}{3}$, the only tuples satisfying the constraints are $(1,3,4)$ and $(2,3,4)$, which correspond to the same vector vector $x=[-1~-1~-1]^\top$. One can check that ${A}\otimes \bar{x}=\emph{\textbf{0}}$. Consequently, we have $\idxset(\bar{A},\frac{1}{3})=\{(1,3,4),(2,3,4)\}$ which implies that $\solset(A,\frac{1}{3})=\{[-1~-1~-1]^\top \}$.
        \item[(ii)] For $\omega\in \{\frac{2}{3},1\}$, there is no tuple that satisfies the requirements. Therefore, $\idxset=\solset=\emptyset$.
    \end{itemize}
\end{exmp}
\begin{exmp}
\label{ex:fully active-indistinct}
    Suppose we have a normalized system \eqref{eq:problem} for
    \[
    A=\begin{bmatrix}
        -3 & 2 & 6\\
        -3 & 4 & 3\\
        5 & 4 & 0
    \end{bmatrix} ~\text{and}~
    \omega = \frac{2}{3}.
    \]
     The principal-order matrix for $A$ is
    \[
    \bar A = \begin{bmatrix}
        1&1&3\\
        1&2&2\\
        3&2&1
    \end{bmatrix}.
    \]
    Furthermore, from $\bar A$, we obtain $f_1=f_2=2$ and $f_3=1$. By Proposition~\ref{prop:candidate-solutions-indistinct}, if $(i_1,i_2,i_3)\in \idxset(\bar{A},\frac{2}{3})$ then \[\emph{\texttt{idx}}_1(\bar{A},\bar{A}(i_1,1))\cup \emph{\texttt{idx}}_2(\bar A,\bar A(i_2,2)) \cup \emph{\texttt{idx}}_3(\bar{A},\bar{A}(i_3,3))=\{1,2,3\}\] and $4\leq \bar{A}(i_1,1)+\bar A(i_2,2)+ \bar A(i_3,3)\leq  6$. Since $\bar{A}(1,1)=\bar A(2,1)$ and $\bar A(2,2)=\bar{A}(3,2)$, we can restrict $i_1\neq 2$ and $i_2\neq 3$. Without providing the details, the are 4 tuples satisfying the constraints: 
    \[
    (1,2,1),(1,2,2),(1,2,3),(3,1,2),
    \]
    which respectively corresponds to vectors
    \[
    \bar{x}\in \left\{
    \begin{bmatrix}
        -1\\-2\\-3
    \end{bmatrix},
    \begin{bmatrix}
        -1\\-2\\-2
    \end{bmatrix},
    \begin{bmatrix}
        -1\\-2\\-1
    \end{bmatrix},
    \begin{bmatrix}
        -3\\-1\\-2
    \end{bmatrix}
    \right\}.
    \]
    All above vectors satisfy $\bar{A}\otimes_\frac{2}{3} \bar{x}=\emph{\textbf{0}}$. Hence, from the same tuples, we have found four fully active solutions for $A\otimes_\frac{2}{3}x=\emph{\textbf{0}}$
    \[
    x\in
    \left\{
    \begin{bmatrix}
        3\\-4\\-6
    \end{bmatrix},
    \begin{bmatrix}
        3\\-4\\-3
    \end{bmatrix},
    \begin{bmatrix}
        3\\-4\\0
    \end{bmatrix},
    \begin{bmatrix}
        -5\\-2\\-3
    \end{bmatrix}
    \right\}\subseteq \solset(A,\frac{2}{3}).
    \]
\end{exmp}

\section{Relaxations}
We recall that, for $\ceils{\omega n}=n$ (max-plus case), if the principal vector $\bar{x}$  \eqref{eq:principal-solution} is a solution for \eqref{eq:problem} then it is the greatest solution \cite{Butkovic2003}: if $x$ is also a solution, then $x\leq \bar{x}$. Likewise, for $\ceils{\omega n}=1$ (min-plus case), such principal vector is the least solution. Inspired by those, we provide procedures to generate other solutions for \eqref{eq:problem} by applying ``relaxation'' to a fully active solution $x$, i.e., increasing or decreasing some elements of $x$. Such modification may produce another fully active solution or even non fully active one. Note that for general $\omega$ we need to consider the possibility that some of the components of $x$ are increased and some of the components of $x$ are decreased, which motivates the following definition. 

Consider subsets $Q,R\subseteq N$ with $Q\cap R=\emptyset$. We say that a fully active solution $x$ admits $(Q,R)$-relaxation, if there exists an $\epsilon>0$ such that every $\hat{x}=[\hat x_1~\cdots~\hat x_n]^\top$, where 
$x_j-\epsilon\leq \hat x_j< x_j$ for $j\in Q$, $x_k+\epsilon\geq \hat x_k>x_k$ for $k\in R$,  
and $\hat x_l=x_l$ for $l\in N\backslash (Q\cup R)$.

Suppose $C$ is a matrix whose entries are $C(i,j)=A(i,j)+x_j$, for all $i\in M$ and $j\in N$. Let us also introduce the following notations
\begin{equation}
\label{e:SiQiRi}
S_i=\{j\in N\colon C(i,j)=0\},\quad  Q_i=Q\cap S_i,\quad R_i=R\cap S_i,
\end{equation}
for $i\in M $. Furthermore, we denote by $q_i$ the number of 
negative 
elements in the $i$-th row of $C$ and by $r_i$ the number of 
positive 
elements in the same row.
As before, we have $p=\ceils{\omega n}$.
 
\begin{prop}
 \label{prop:relax-solution}
 Let $Q,R\subseteq N$ be such that $Q\cap R=\emptyset$. Suppose $x=[x_1~\cdots~x_n]^\top$ is a fully active solution for normalized linear equation \eqref{eq:problem} and $C$ is a matrix defined as $C(i,j)=A(i,j)+x_j$. Then $x$ admits $(Q,R)$-relaxation if and only if the following two conditions hold for all $i\in M$:
\begin{itemize}
        \item[{\rm (i)}] $|Q_i|\leq p-1-q_i$, 
        \item[{\rm (ii)}] $|R_i|\leq n-p-r_i$ 
\end{itemize}
\end{prop}
\begin{proof}
 Conditions (i) and (ii) are sufficient to ensure that, after 
 decreasing 
 the components in $Q$ and 
 increasing 
 the components in $R$ and modifying matrix $C$ accordingly, 
\begin{itemize}
\item[(i)] there are less than $p$ negative elements in $C(i,\cdot)$ (using the first condition of the theorem),
\item[(ii)] there are no more than $n-p$ positive elements in $C(i,\cdot)$
(using the second condition of the theorem).
\end{itemize}
Also observing that by choosing small enough $\epsilon$ the signs of all nonzero entries of $C$ remain unchanged, we obtain that the $p$-th smallest element is $0$ in each row of the (modified) matrix $C$, thus $A\otimes_{\omega}\hat{x}=\textbf{0}$.

The conditions are also necessary. Indeed, if the first one does not hold, then in some row there are more than 
$p-1$ negative elements 
in the modified matrix $C$ implying that the $p$-th smallest element in that row of $C$ becomes 
negative, 
a contradiction. Similarly, if the second condition does not hold, then the $p$-th smallest element becomes 
positive 
in one of the rows of $C$.
\end{proof}
 
Observe that if we add up the inequalities in conditions (i) and (ii) of Proposition~\ref{prop:relax-solution}, then we obtain $|Q_i|+|R_i|\leq n-q_i-r_i-1$, which is the same as 
\begin{equation}
\label{e:onezeroremains}
    |Q_i|+|R_i|\leq |S_i|-1.
\end{equation}
In words, we obtain a natural condition that at least one zero should remain in each row of $C$ after a relaxation.

Important special cases of $(Q,R)$-relaxations appear when $Q=\emptyset$ (thus we 
increase 
a number of components simultaneously) and $R=\emptyset$
(so we 
decrease 
a number of components simultaneously). Further if $Q=\emptyset$, $R=\{j\}$ and $x$ admits such $(Q,R)$-relaxation then $j$ is called \textit{increasable} 
with respect to $x$. Similarly if $Q=\{j\}$, $R=\emptyset$ and $x$ admits such $(Q,R)$-relaxation then $j$ is called 
\textit{decreasable} 
with respect to $x$. We then have two important corollaries of Proposition~\ref{prop:relax-solution}

\begin{cor}
    \label{cor:decrease-solution}
    An index $j\in N$ is decreasable with respect to a fully active solution $x$ if and only if the following conditions hold: for each $i\in N$ such that $C(i,j)=0$
    \begin{itemize}
        \item[{\rm (i)}] there are at least two zero elements at $C(i,\cdot)$,
        \item[{\rm (ii)}] $q_i\leq p-2$.
    \end{itemize}
\end{cor}
\begin{proof} 
We need to show that the conditions of the present claim are equivalent to the conditions of Proposition~\ref{prop:relax-solution}, for the special case which is considered. 

The first condition of the present claim can be expressed as $n-r_i-q_i\geq 2$. So for each $i$ with $C(i,j)=0$  we have $n-r_i-q_i\geq 2$ and $q_i\leq p-2$. It is easy to see that both of these inequalities are implied by $n-p-r_i\geq 0$ and $p-1-q_i\geq 1$ is obtained from the first condition of Proposition~\ref{prop:relax-solution}. On the other hand, $p-1-q_i\geq 1$ is the same as $q_i\leq p-2$, while $n-p-r_i\geq 0$ follows since $x$ is a solution.
\end{proof}
  
 \begin{cor}
   \label{cor:increase-solution}
     An index $j\in N$ is increasable with respect to a fully active solution $x$ if and only if the following conditions hold: for each $i\in N$ such that $C(i,j)=0$
    \begin{itemize}
        \item[{\rm (i)}] there are at least two zero elements in $C(i,\cdot)$,
        \item[{\rm (ii)}] $r_i\leq n-p-1$ 
    \end{itemize}  
 \end{cor}
\begin{proof} 

The first condition of the present claim can be expressed as $n-r_i-q_i\geq 2$. So for each $i$ with $C(i,j)=0$  we have $n-r_i-q_i\geq 2$
and $r_i\leq n-p-1$. It is easy to see that both of these inequalities are implied by $p-1-q_i\geq 0$ and $n-p-r_i\geq 1$ obtained from the second condition of Proposition~\ref{prop:relax-solution}. On the other hand, $n-p-r_i\geq 1$ is the same as $r_i\leq n-p-1$, while $p-1-q_i\geq 0$ follows since $x$ is a fully active solution.
\end{proof}

Note that the above results do not describe by how much $x_j$ can be increased or decreased. 
By taking the largest increment and decrement for $x_j$, we then express the ``possibly relaxed solutions'' w.r.t a fully active solution $x$ as a vector of intervals
\begin{equation}
\label{eq:relaxed-solutions}
\mathrm{rel}(x)=
    \begin{bmatrix}
        [x_1+\delta_1, x_1+\epsilon_1]\\
        [x_2+\delta_2, x_2+\epsilon_2]\\
        \vdots\\
        [x_n+\delta_n, x_n+\epsilon_n]
    \end{bmatrix},
\end{equation}
where $\delta_j\leq 0$ and $\epsilon_j\geq 0$ for $j\in N$. Indeed, if $j$ is neither increasable nor decreasable, then $\delta_j=\epsilon_j=0$. In some cases, it is possible that there are multiple increasable or decreasable indices but the corresponding variables cannot be modified at the same time. Hence, the set of relaxed solutions in general is not the same as \eqref{eq:relaxed-solutions}, and we also define $\mathrm{Rel}(x)$ as the set of all possible relaxed solutions w.r.t. $x$.

\if{
\cref{cor:min-plus-increasable,cor:max-plus-decreasable} describe the conditions of decreasable and increasable indices for linear equations in min-plus and max-plus algebra and thus are closely related to the known results described, e.g., in~\cite{Butkovic2010}. \cref{cor:min-plus-increasable} relates the requirements for increasable indices with the appearance of 1 in each row of $\bar{A}$. On the other hand, \cref{cor:max-plus-decreasable} associates the conditions for decreasable indices with the appearance of $m$ in each row of $\bar{A}$. \tred{Interestingly, they also relate the relaxation conditions for the cases that are ``close'' to min-plus and max-plus algebra, i.e., when $\ceils{\omega n}=2$ and $\ceils{\omega n}=n-1$ respectively.}

Finally, \cref{cor:square-not-increasable-not-decreasable} shows that the fully active solutions for \eqref{eq:problem} when $m=n$ and the matrix is column distinct cannot be relaxed. Hence, number of solutions is guaranteed to be finite. In max-plus and min-plus case, this corresponds to the case where the right hand side $b$ of $A\otimes x=\textbf{b}$  belongs to the simple image set of $A$ and the solution of this system is unique~\cite{Butkovic2010}.
}\fi

In the next corollary we observe some limitations on the possibilities to increase or decrease the components of a fully active solution in certain cases of $p=\ceils{\omega n}$.

\begin{cor}
\label{cor:limitations}
Consider a normalized system \eqref{eq:problem} where $A\in \R^{m\times n}$, and let $x$ be a fully active solution of \eqref{eq:problem}.
Then
\begin{itemize}
\item[{\rm (i)}] if $p=1$, then there are no decreasable indices $j\in N$;
\item[{\rm (ii)}] if $p=2$ and if there are different decreasable indices $j$ and $k$ such that $C(i,j)=C(i,k)=0$ for some $i\in N$, then the corresponding components of $x$ cannot be decreased together;
\item[{\rm (iii)}] if $p=n$, then there are no increasable indices $j\in N$;
\item[{\rm (iv)}] if $p=n-1$ and if there are different increasable indices $j$ and $k$ such that $C(i,j)=C(i,k)=0$ for some $i\in N$, then the corresponding components of $x$ cannot be increased together.
\end{itemize}
\end{cor}
\begin{proof}
 (i) 
 As $p=1$, it is straightforward 
 to see that there is no negative element at each row of $C$ (where $C$ is defined as before by $C(i,j)=A(i,j)+x_j$ for $i\in M$ and $j\in N$). Hence, the second condition of Corollary~\ref{cor:decrease-solution} is not satisfied. This implies that there are no decreasable indices.\\
(ii) Now suppose that $\ceils{\omega n}=2$. By Corollqry~\ref{cor:decrease-solution}, if $j$ is a decreasable index and $i$ is such that  $C(i,j)=0$, then there is no negative element in $C(i,\cdot)$. Hence, one can only decrease $\ceils{\omega n}-1=1$ components of $x$ at once.\\
(iii) Similar to (i) (using Corollary~\ref{cor:increase-solution}). \\
(iv) Similar to (ii) (using Corollary~\ref{cor:increase-solution}).
\end{proof}

In the case of column distinct matrices, when $m=n$, we have already seen that no relaxation is possible.

\begin{cor}
\label{cor:square-not-increasable-not-decreasable}
Consider a normalized system \eqref{eq:problem} where $A\in \R^{m\times n}$ is a column distinct matrix and suppose that $\bar A$ is the corresponding principal-order matrix for $A$. If $m=n$, then for each $\omega$ and the corresponding fully active solutions $x\in\solset$ all indices $j\in N$ are neither decreasable nor increasable. Consequently, $|\solset|$ is finite for each $\omega$.
\end{cor}
\begin{proof}
    This is a direct consequence of Corollary~\ref{cor:m=n-cases}, part (iv).
\end{proof}

Example~\ref{ex:relaxation-distinct} below demonstrates the steps to find all solutions of \eqref{eq:problem} by applying relaxation when the matrix is column distinct. It should be noted that these relaxation steps still work for general case when the matrix is not necessarily column distinct. For the general case we also present Example~\ref{ex:all-relaxed-indistinct} below.
 
\begin{exmp}
    \label{ex:relaxation-distinct}
    Consider the linear equation in Example~\ref{ex:distinct-3x4}. We will demonstrate the steps to obtain solutions that are not fully active by applying relaxation
    \begin{itemize}
        \item[\emph{(i)}] For $\omega=\frac{1}{4}$, the only fully active solution is ${x}=[-2~1~2~-1]^\top$. 
        Applying Corollary~\ref{cor:increase-solution}, we find that increasable indices are $1$ and $4$, with $\epsilon_1=\epsilon_4=+\infty$. However, $x_1$ and $x_4$ cannot be increased at the same time (since the second row of a matrix $C$ defined as $C(i,j)=A(i,j)+x_j$ has two positive elements already). It can be checked that
        \[
        \solset(A,\frac{1}{4})=\mathrm{Rel}(x)=
        \begin{bmatrix}
            [-2,+\infty]\\
            [1,1]\\
            [2,2]\\
            [-1,-1]
        \end{bmatrix}\cup
        \begin{bmatrix}
            [-2,-2]\\
            [1,1]\\
            [2,2]\\
            [-1,+\infty]
        \end{bmatrix}.
        \]
        \item[\emph{(ii)}] For $\omega=\frac{1}{2}$, suppose we take $x=[-2~1~2~-3]^\top,y=[-2~1~-6~-3]^\top$ and $C_1,C_2$ defined as $C_1(i,j)=A(i,j)+x_j, C_2(i,j)=A(i,j)+y_j$:
        \[
        C_1=\begin{bmatrix}
            3&6&0&0\\
            0&5&8&-2\\
            4&0&9&-1
        \end{bmatrix}\!,~
        C_2=\begin{bmatrix}
            3&6&-8&0\\
            0&5&0&-2\\
            4&0&4&-1
        \end{bmatrix}\!.
        \]
        Based on $C_1$ and applying Corollaries \ref{cor:increase-solution} and \ref{cor:decrease-solution}, for $x$, one can only decrease the value for $x_3$ and $x_4$, but not simultaneously (applying Proposition~\ref{prop:relax-solution}). The biggest decrement for $x_3$ is $|\delta_3|=8$: if $\delta_3<-8$ then there will be two negative elements on second column of $C_1$. For $x_4$, the largest decrement is $|\delta_4|=+\infty$. 
        
        On the other hand, based on $C_2$, both $y_1$ and $y_3$ can be increased (again, not simultaneously since the second row of $C_2$ has one positive element). The corresponding biggest increments are $\epsilon_1=+\infty$ and $\epsilon_3=8$. As a result, we have
        \[\mathrm{Rel}(x)\!=\!
        \begin{bmatrix}
            [-2,-2]\\
            [1,1]\\
            [-6,2]\\
            [-3,-3]
        \end{bmatrix}\cup
        \begin{bmatrix}
            [-2,-2]\\
            [1,1]\\
            [2,2]\\
            [-\infty,-3]
        \end{bmatrix}\!,
        \mathrm{Rel}(y)\!=\!
        \begin{bmatrix}
            [-2,+\infty]\\
            [1,1]\\
            [-6,-6]\\
            [-3,-3]
        \end{bmatrix}\cup
        \begin{bmatrix}
            [-2,-2]\\
            [1,1]\\
            [-6,2]\\
            [-3,-3]
        \end{bmatrix}.
        \]
        Similarly, for $v$ and $w$ the resulting relaxation sets are
        \[\mathrm{Rel}(
        v
        )\!=\!
        \begin{bmatrix}
            [-6,-6]\\
            [-4,1]\\
            [2,2]\\
            [-1,-1]
        \end{bmatrix}\cup
        \begin{bmatrix}
            [-6,-6]\\
            [-4,1]\\
            [2,2]\\
            [-1,+\infty]
        \end{bmatrix}\!,
        \mathrm{Rel}(
        w
        )\!=\!
        \begin{bmatrix}
            [-\infty,-6]\\
            [1,1]\\
            [2,2]\\
            [-1,-1]
        \end{bmatrix}\cup
        \begin{bmatrix}
            [-6,-6]\\
            [-4,1]\\
            [2,2]\\
            [-1,-1]
        \end{bmatrix}\!.
        \]
        where $v=\begin{bmatrix}
            -6&-4&2&-1
        \end{bmatrix}^\top$ and $w=\begin{bmatrix}
            -6&1&2&-1
        \end{bmatrix}^\top$. Notice that, $\mathrm{Rel}(x)\cap\mathrm{Rel}(y)\neq \emptyset$ and $\mathrm{Rel}(v)\cap \mathrm{Rel}(w)\neq \emptyset$. Furthermore, the set of solutions $\solset(A,\frac{1}{2})$ can be expressed as the union of six sets
        \begin{align*}
            \solset(A,\frac{1}{2})=&
         \begin{bmatrix}
            [-2,-2]\\
            [1,1]\\
            [-6,2]\\
            [-3,-3]
        \end{bmatrix}\cup
        \begin{bmatrix}
            [-2,-2]\\
            [1,1]\\
            [2,2]\\
            [-\infty,-3]
        \end{bmatrix}\cup
        \begin{bmatrix}
            [-2,+\infty]\\
            [1,1]\\
            [-6,-6]\\
            [-3,-3]
        \end{bmatrix}\cup\\
        &
        \begin{bmatrix}
            [-6,-6]\\
            [-4,1]\\
            [2,2]\\
            [-1,-1]
        \end{bmatrix}\cup
        \begin{bmatrix}
            [-6,-6]\\
            [-4,1]\\
            [2,2]\\
            [-1,+\infty]
        \end{bmatrix}\cup\begin{bmatrix}
            [-\infty,-6]\\
            [1,1]\\
            [2,2]\\
            [-1,-1]
        \end{bmatrix}\!.
        \end{align*}
        \item[\emph{(iii)}] Without providing the details, for $\omega=\frac{3}{4}$, the resulting relaxation sets are
        \begin{align*}
        \mathrm{Rel}(
        p
        )\!&=\!
        \begin{bmatrix}
            [-6,-6]\\
            [-4,-4]\\
            [-6,2]\\
            [-3,-3]
        \end{bmatrix}\cup
        \begin{bmatrix}
            [-6,-6]\\
            [-4,-4]\\
            [2,2]\\
            [-\infty,-3]
        \end{bmatrix}\!,
        \mathrm{Rel}(q)\!=\!
        \begin{bmatrix}
            [-6,-6]\\
            [-4,1]\\
            [-6,-6]\\
            [-3,-3]
        \end{bmatrix}\cup
        \begin{bmatrix}
            [-6,-6]\\
            [-4,-4]\\
            [-6,2]\\
            [-3,-3]
        \end{bmatrix},\\
        \mathrm{Rel}(
        r
        )\!&=\!
        \begin{bmatrix}
            [-\infty,-6]\\
            [1,1]\\
            [-6,-6]\\
            [-3,-3]
        \end{bmatrix}\cup
        \begin{bmatrix}
            [-6,-6]\\
            [-4,1]\\
            [-6,-6]\\
            [-3,-3]
        \end{bmatrix}.
        \end{align*}
        where $p=\begin{bmatrix}
            -6&-4&2&-3
        \end{bmatrix}^\top , q=\begin{bmatrix}
            -6&1&-6&-3
        \end{bmatrix}^\top$, and \\$r=\begin{bmatrix}
            -6&-4&-6&-3
        \end{bmatrix}^\top$. The solution set $\solset(A,\frac{3}{4})$ can be expressed as the union of four sets
        \[
        \solset(A,\frac{3}{4})=\begin{bmatrix}
            [-6,-6]\\
            [-4,-4]\\
            [-6,2]\\
            [-3,-3]
        \end{bmatrix}\cup
        \begin{bmatrix}
            [-6,-6]\\
            [-4,-4]\\
            [2,2]\\
            [-\infty,-3]
        \end{bmatrix}\cup\begin{bmatrix}
            [-\infty,-6]\\
            [1,1]\\
            [-6,-6]\\
            [-3,-3]
        \end{bmatrix}\cup
        \begin{bmatrix}
            [-6,-6]\\
            [-4,1]\\
            [-6,-6]\\
            [-3,-3]
        \end{bmatrix}\!.
        \]
    \end{itemize}
\end{exmp}

\begin{exmp}
\label{ex:all-relaxed-indistinct}
    Let us reconsider the normalized system in Example~\ref{ex:fully active-indistinct}. Without giving the the details, the relaxed solution for each $x$ is
    \begin{align*}
        \mathrm{Rel}\left(\begin{bmatrix}
        3\\-4\\-6
    \end{bmatrix} \right)&=
    \begin{bmatrix}
        [3, +\infty]\\
        [-4, -4]\\
        [-6,-6]
    \end{bmatrix}\cup
    \begin{bmatrix}
        [3, 3]\\
        [-4, -4]\\
        [-6,0]
    \end{bmatrix},\hspace{2ex}
    \mathrm{Rel}\left(\begin{bmatrix}
        3\\-4\\-3
    \end{bmatrix} \right)=
    \begin{bmatrix}
        [3, 3]\\
        [-4, -4]\\
        [-6,0]
    \end{bmatrix},\\
    \mathrm{Rel}\left(\begin{bmatrix}
        3\\-4\\0
    \end{bmatrix}\right)&=
    \begin{bmatrix}
        [3, 3]\\
        [-\infty, -4]\\
        [-6,-6]
    \end{bmatrix}\cup
    \begin{bmatrix}
        [3, 3]\\
        [-4, -4]\\
        [-6,0]
    \end{bmatrix},\hspace{1.5ex}
     \mathrm{Rel}\left(\begin{bmatrix}
        -5\\-2\\-3
    \end{bmatrix}\right)=
    \begin{bmatrix}
        [-5, -5]\\
        [-2, -2]\\
        [-3,-3]
    \end{bmatrix}\!\!.
    \end{align*}
Notice that the second relaxation set is a subset of the first and the third sets. Furthermore, the last set contains a single vector. Finally, the set of solutions can be expressed as the union of four sets
\[
\solset(A,\frac{2}{3})=\begin{bmatrix}
        [3, +\infty]\\
        [-4, -4]\\
        [-6,-6]
    \end{bmatrix}\cup \begin{bmatrix}
        [3, 3]\\
        [-\infty, -4]\\
        [-6,-6]
    \end{bmatrix}\cup
    \begin{bmatrix}
        [3, 3]\\
        [-4, -4]\\
        [-6,0]
    \end{bmatrix}\cup\begin{bmatrix}
        [-5, -5]\\
        [-2, -2]\\
        [-3,-3]
    \end{bmatrix}\!\!.
\]
\end{exmp}

Examples \ref{ex:relaxation-distinct}-\ref{ex:all-relaxed-indistinct} only demonstrate the cases when the variables must be relaxed separately. The following example showcases the relaxation of variables that can be done simultaneously.

\begin{exmp}
      Suppose we have a normalized linear equation \eqref{eq:problem} where
    \[
A=\begin{bmatrix}
    2&-1&7&-3\\
    2&5&2&0\\
    2&6&3&2\\
    2&-1&6&4
\end{bmatrix}~\text{and}~\omega=\frac{1}{2}.
\]
One could check that $y=[-2~1~-3~3]^\top$ is one of the solutions. Let us define a matrix $C$ where $C(i,j)=A(i,j)+x_j$ i.e.,
    \[
    C=
    \begin{bmatrix}
        0&0&4&0\\
        0&6&-1&3\\
        0&7&0&5\\
        0&0&3&7
    \end{bmatrix}
    \]
    Notice that, based on matrix $C$, the only increasable index w.r.t. $x$ is 4. The corresponding largest increment is $\epsilon_4=+\infty$. On the other hand, the decreasable indices w.r.t $x$ are 2,3, and 4. 
    Proposition~\ref{prop:relax-solution} gives the following possible $(Q,R)$-relaxations: 1) $Q=\{2\}$, $R=\{4\}$, 2) $Q=\{3\}$, $R=\{4\}$, 
3) increase or decrease $x_4$ only, 4) $Q=\{2,3\}$, $R=\{4\}$, 5) $Q=\{3,4\}t$, $R=\emptyset$. Without giving the details, the corresponding relaxations for each of these opportunities are
    \begin{align*}
    \begin{bmatrix}
        [-2,-2]\\
        [-5,1]\\
        [-3,-3]\\
        [3,+\infty]
    \end{bmatrix}, 
    \begin{bmatrix}
        [-2,-2]\\
        [1,1]\\
        [-\infty,-3]\\
        [3,+\infty]
    \end{bmatrix},
    \begin{bmatrix}
        [-2,-2]\\
        [1,1]\\
        [-3,-3]\\
        [0,+\infty]
    \end{bmatrix},
    \begin{bmatrix}
        [-2,-2]\\
        [-5,1]\\
        [-6,-3]\\
        [3,+\infty]
    \end{bmatrix},
    \begin{bmatrix}
        [-2,-2]\\
        [1,1]\\
        [-7,-3]\\
        [0,3]
    \end{bmatrix}.
    \end{align*}
    Notice that, the first relaxation vector is subset of the fourth one. 
    Hence, $\mathrm{Rel}(x)$ is sufficient to be expressed as the union of the second, the third, the fourth and the fifth sets.
\end{exmp}

Corollary~\ref{ex:relaxation-distinct} suggests that, for column distinct cases and when $m<n$, all fully active solutions can be relaxed. Proposition~\ref{prop:all-relaxed} formally proves this condition. In fact, it also works on general cases. On the other hand, Example~\ref{ex:all-relaxed-indistinct} may also suggest that all fully active solutions can be relaxed (even when $m=n$). However, this is not true in general. The condition of $|\solset|=1$ could happen when $\ceils{\omega n}\in \{1,n\}$ as in Proposition~\ref{prop:uniquenes-butkovic}. Unlike column distinct cases, the condition for relaxed solution does not only depend on the dimension of the matrix but also on the sum of $f_j$ defined in \eqref{eq:largest-freq}.

\begin{prop}
    \label{prop:all-relaxed}
    Suppose we have a normalized \eqref{eq:problem} where $A\in \R^{m\times n}$ is a column distinct matrix and $m<n$. Then, each fully active solution $x\in \solset$ can be relaxed.
\end{prop}
\begin{proof}
    Let us assume there exists a fully active solution ${x}\in \solset$ which cannot be relaxed. Define matrix $C$ as $C(i,j)=A(i,j)+x_j$. Since all indices $j\in N$ are neither increasable nor decreasable, by Corollary~\ref{cor:increase-solution} and Corollary~\ref{cor:decrease-solution}, one of the following conditions must hold: in each $i\in M$,
    \begin{itemize}
        \item[(i)] there is exactly one zero at $C(i,\cdot)$,
        \item[(ii)] there are exactly $p-1$ negative elements and $n-p$ positive elements at $C(i,\cdot)$ where $p=\ceils{\omega n}$
    \end{itemize}
    Notice that the second condition implies the first one. Hence, to show contradiction, it is sufficient that the first requirement cannot hold. Since $A$ is column distinct, there is exactly one zero in each column of $C$. As a result, there are exactly $n$ zero elements in $C$. However, since $m<n$, there must at least one row of $C$ with multiple zero elements. 
\end{proof}

The following corollary lists the possible numbers of solutions depending the size of the matrix for the column distinct cases. The arguments are based on Corollaries \ref{cor:m>n-cases}-\ref{cor:m=n-cases} and Proposition~\ref{prop:all-relaxed} and are omitted.
\begin{cor}
    \label{cor:num-sol-distinct}
    Suppose we have a normalized \eqref{eq:problem} where $A\in \R^{m\times n}$ is a column distinct matrix. Then, the following condition holds
    \begin{itemize}
        \item[{\rm (i)}] if $m>n$, then $|\solset|=0$,
        \item[{\rm (ii)}] if $m=n$, then $0\leq |\solset|<+\infty$,
        \item[{\rm (iii)}] if $m<n$, then $|\solset|\in \{0,+\infty\}$.
    \end{itemize}
\end{cor}


Finally, Corollary~\ref{cor:num-sol-indistinct} presented below is the generalization of Corollary~\ref{cor:num-sol-distinct} by taking account the maximum number of duplicates in each row of the normalized matrix in \eqref{eq:problem}. We recall that, the column distinct cases imply $f_1+\ldots+f_n=n$. Part (i) is due to Corollaries~\ref{cor:m>n-cases} and~\ref{cor:m>f} while part (ii) is due to Corollaries~\ref{cor:m=n-cases} and~\ref{cor:m=f-cases-indistinct}. Part (iv) is due to Proposition~\ref{prop:all-relaxed}. Lastly, part (iii) is the unique case which only happens when the matrix is square and not column distinct. The uniqueness of solution for \eqref{eq:problem} is possible when $\ceils{\omega n}\in \{1,n\}$. Furthermore,  as demonstrated in Example~\ref{ex:all-relaxed-indistinct}, $|\solset|=+\infty$ is also a possibility.

\begin{cor}
    \label{cor:num-sol-indistinct}
    Suppose we have a normalized \eqref{eq:problem} 
    and let us define $f=f_1+\ldots+f_n$. Then the following claims are true:
    \begin{itemize}
        \item[{\rm (i)}] if $m>f\geq n$, then $|\solset|=0$,
        \item[{\rm (ii)}] if $m=f\geq n$, then $0\leq |\solset|<+\infty$,
        \item[{\rm (iii)}] if $m=n<f$, then $0\leq |\solset|\leq  +\infty$,
        \item[{\rm (iv)}] if $m<n\leq f$, then $ |\solset|\in \{0, +\infty\}$.
    \end{itemize}
\end{cor}

\if{

}\fi


\begin{thebibliography}{10}

\bibitem{baccelli92}
F.~Baccelli, G.~Cohen, G.~J. Olsder, and J.-P. Quadrat.
\newblock {\em Synchronization and linearity: an algebra for discrete event
  systems}.
\newblock John Wiley \& Sons Ltd, 1992.

\bibitem{Butkovic2003}
P.~Butkovi{\v{c}}.
\newblock Max-algebra: the linear algebra of combinatorics?
\newblock {\em Linear Alg. Appl.}, 367:313--335, 2003.

\bibitem{Butkovic2010}
P.~Butkovi{\v{c}}.
\newblock {\em Max-linear systems: theory and algorithms}.
\newblock Springer, London, 2010.

\bibitem{cassandras2008}
C.~Cassandras and S.~Lafortune, editors.
\newblock {\em Introduction to discrete event systems}, Boston, MA, 2008.
  Springer.

\bibitem{cuninghamegreen}
R.~A. Cuninghame-Green.
\newblock Minimax algebra.
\newblock {\em Lecture Notes in Economics and Mathematical Systems}, 166, 1979.

\bibitem{dotoli2006}
M.~Dotoli and M.~Fanti.
\newblock An urban traffic network model via coloured timed {Petri} nets.
\newblock {\em Control Engineering Practice}, 14(10):1213--1229, 2006.

\bibitem{GG-98}
S.~Gaubert and J.~Gunawardena.
\newblock The duality theorem for min-max functions.
\newblock {\em C.R.A.S. Paris, S{\'e}rie I}, 326:43--48, 1998.

\bibitem{Gun1994}
J.~Gunawardena.
\newblock Min-max functions.
\newblock {\em Discrete Event Dynamic Systems}, 4(4):377--407, 1994.

\bibitem{Heid2006}
B.~Heidergott, G.-J. Olsder, and J.~van~der Woude.
\newblock {\em Max-plus at work}.
\newblock Princeton University Press, 2006.

\bibitem{imaev2008}
A.~Imaev and R.~Judd.
\newblock Hierarchial modeling of manufacturing systems using max-plus algebra.
\newblock In {\em Proc. of the 2008 American Control Conference}, pages
  471--476. IEEE, 2008.

\bibitem{EPthesis}
E.~L. Patel.
\newblock {\em Maxmin-plus models of asynchronous computation}.
\newblock PhD thesis, The University of Manchester (United Kingdom), 2012.

\bibitem{Subiono}
Subiono.
\newblock {\em On classes of min-max-plus systems and their applications}.
\newblock PhD thesis, TU Delft, 2000.

\bibitem{SO-97}
Subiono and G.~Olsder.
\newblock On bipartite min-max-plus systems.
\newblock In {\em Proceedings of the 1997 European control conference}, pages
  766--771, 1997.

\bibitem{Zimmerman2007}
A.~Zimmermann.
\newblock {\em Stochastic discrete event systems}.
\newblock Springer, Berlin, 2007.

\end{thebibliography}

\begin{thebibliography}{00}


\bibitem{baccelli92} Baccelli, F., Cohen, G., Olsder, G.J., \& Quadrat, J.P., Synchronization and linearity: an algebra for discrete event systems, 1992.

\bibitem{cassandras2008} Cassandras, C.G., \& Lafortune, S. (Eds.), \textit{Introduction to discrete event systems}. Boston, MA: Springer US, 2008.

\bibitem{imaev2008}  Imaev, A., \& Judd, R.P., ``Hierarchial modeling of manufacturing systems using max-plus algebra", In 2008 American Control Conference (pp. 471--476), IEEE, June 2008.

\bibitem{dotoli2006} Dotoli, M., \& Fanti, M.P. (2006). ``An urban traffic network model via coloured timed Petri nets". \textit{Control Engineering Practice}, \textbf{14}(10), 1213--1229, 2006.

\bibitem{Zimmerman2007}  Zimmermann, A., \textit{Stochastic discrete event systems}. Springer, Berlin Heidelberg New York, 2007.

\bibitem{cuninghamegreen} Cuninghame-Green R.A., ``Minimax Algebra", in \textit{Lecture Notes in Economics and Mathematical Systems}, \textbf{166}, Springer, Berlin, 1979.

\bibitem{Butkovic2010} Butkovi\v{c} P., ``Max-Linear Systems: Theory and Algorithms", in \textit{Springer Monographs in Mathematics}, Springer-Verlag, London, 2010.

\bibitem{Butkovic2003} Butkovi\v{c} P., ``Max-algebra: the linear algebra of combinatorics?". \emph{Linear Algebra and its Applications}, \textbf{367}(313--335), doi = \url{https://doi.org/10.1016/S0024-3795(02)00655-9}, 2003.

\bibitem{Gun1994} Gunawardena, J., ``Min-max functions". \textit{Discrete Event Dynamic Systems}, \textbf{4}(4), pp.377-407, 1994.

\bibitem{Heid2006} Heidergott, B., Olsder, G.J., Van Der Woude, J., \& van der Woude, J.W.. \textit{Max Plus at work: modeling and analysis of synchronized systems: a course on Max-Plus algebra and its applications (Vol. 13)}. Princeton University Press, 2006.

\bibitem{EPthesis} Patel, Ebrahim L. Maxmin-plus models of asynchronous computation. The University of Manchester (United Kingdom), 2012.

\end{thebibliography}
\end{document}